\documentclass[11pt,oneside,reqno]{amsart}
\usepackage[english]{babel}
\usepackage{amsmath,amsfonts,amsthm,amssymb,amsbsy,upref,color,graphicx,amscd,hyperref,enumerate,bbm,comment, algorithm, algorithmic,dsfont}
\hypersetup{colorlinks,breaklinks,
            linkcolor=[rgb]{0,0,0},
            citecolor=[rgb]{0,0,0},
            urlcolor=[rgb]{0,0,0}}

\usepackage{xfrac,xcolor}
\usepackage{mathtools,url}
\usepackage[active]{srcltx}
\usepackage{cases} 
\usepackage[shortlabels]{enumitem}
\usepackage{cancel}
\usepackage{braket}
\usepackage[capitalise, nameinlink, noabbrev]{cleveref} 

\def\N{\mathbb{N}}

\numberwithin{equation}{section}
\theoremstyle{definition}
\newtheorem{theo}{Theorem}[section]
\newtheorem{theorem}[theo]{Theorem}

\newtheorem{lemma}[theo]{Lemma}
\newtheorem{corollary}[theo]{Corollary}

\theoremstyle{remark}

\def\XXint#1#2#3{{\setbox0=\hbox{$#1{#2#3}{\int}$ }
\vcenter{\hbox{$#2#3$ }}\kern-.6\wd0}}

\title{The Gross-Pitaewskii equation with time and space dependent coefficients}
\author[F.~Lai]{Federico Lai}

\address {Federico Lai \newline \indent
Dipartimento di Matematica, Universit\`a di Pisa \newline \indent
Largo Bruno Pontecorvo, 5, 56127 Pisa, Italy}
\email{federico.lai@phd.unipi.it}

\begin{document}


\begin{abstract}
We study the existence of weak solutions of a generalized Gross-Pitaewskii equation, with time and space dependent coefficients that could blow up or vanish asymptotically in time, with initial data not necessarily segregated. We also study the asymptotic behavior in time of these solutions, including cases in which the segregation phenomenon appears.
\end{abstract}

\maketitle
\setcounter{tocdepth}{1}
\tableofcontents

\section{Introduction}
In recent years, the study of the Gross-Pitaewskii equation has become particularly interesting. This is a parabolic system that has the following form:
\begin{equation} \label{gpo1}
\partial_t w_{i} -d_i\Delta w_{i}=f_i(w_{i})-\beta w_{i}\sum_{j\ne i}a_{ij}w^2_{j} \quad \text{in} \, \Omega \times \mathbb{R}^+, \,\, \forall i=1,\dots,k,
\end{equation}
where $\Omega \subset \mathbb{R}^d$ is a bounded and smooth domain, $k \ge 2$ is a fixed integer, $d_i,a_{i,j}>0$ are constants, $f_i$ are functions, $\beta \ge0$ is a parameter, and $w=(w_1,\dots,w_k)$ is the solution that you want to find. \\
This parabolic system is widely used to characterize the dynamics of interaction between $k$ species (see \cite{danwanzha}). Here, $\beta$ is an interaction coefficient, and a larger $\beta$ indicates a less interaction between species $i$ and $j$. Each $w_i(x,t) $ represents the value of species $i$ at point $x$ at time $t$. For these reasons, the Gross-Pitaewskii equation is also coupled with suitable initial data $v_0=(v_{01}, \dots, v_{0k})$, boundary conditions $g=(g_1,\dots,g_k)$, and $L^\infty$ bounds on the solutions. \\
A very interesting problem is to find solutions $w_\beta$ of \eqref{gpo1} for every fixed $\beta$, passing to the limit for $\beta \to +\infty$.
This has been done in \cite{audsertil}, \cite{parab 1}, \cite{danwanzha}, \cite{parab 2}, and \cite{parab 3}: the limit function $w_\infty=(w_{\infty 1}, \dots, w_{\infty,k})$ satisfies the segregation condition:
\begin{equation*}
w_{\infty i}\cdot w_{\infty j}=0 \quad \text{a.e. in} \, \Omega \times \mathbb{R}^+, \, \forall i \ne j.
\end{equation*}
A necessary request to pass to the limit for $\beta \to +\infty$ is that the initial datum $v_0$ must satisfy the segregation condition: $v_{0i} \cdot v_{0j}=0$ a.e. in $\Omega$. \\
In \cite{ellip 1}, \cite{ellip 2}, \cite{ellip 3}, \cite{ellip 4}, \cite{freeb ell 1}, and \cite{freeb ell 2}, the elliptic counterpart of this parabolic problem has been studied. \\
In this work we study a generalized Gross-Pitaewskii equation, coupled with initial and boundary data, of this form:
\begin{equation} \label{gpg1}
\begin{cases}
r_i(x)\partial_t w_{i} -d_i(t)\Delta w_{i}=f_i(t,x,w_{i})-\beta(t,x)w_{i}\sum_{j \ne i}a_{ij}(t,x)w^2_{j} \quad \text{in} \, \Omega \times \mathbb{R}^+ \quad \forall i=1,\dots,k, \\
w_{i}(0,x)=v_{0,i}(x) \quad \text{in} \,\, \Omega, \quad \forall i=1, \dots,k, \\
w_{i}=g_i \quad \text{on} \,\, \partial \Omega, \quad \forall i=1,\dots,k, \\
0 \le w_{i} \le 1 \quad \forall i=1,\dots,k,
\end{cases}
\end{equation}
where $r_i$, $d_i$, $f_i$, $\beta$, $a_{ij}$ are coefficients depending on time and/or space that, for $t \to +\infty$, could vanish or blow up. \\
The main difference between \eqref{gpo1} and \eqref{gpg1} is that, in the latter, $\beta$ is a function, that could blow up asymptotically in time.
Thanks to the fact that, for bounded times, $\beta$ is bounded, we have the flexibility to choose the initial datum $v_0$ not necessarily segregated. This is a method to asymptotically segregate a function over time. It is reasonable to consider dynamics where segregation appears later, such as in human population dynamics. \\
In \eqref{gpg1} also the other coefficients are time and/or space dependent, including cases in which they blow up or vanish asymptotically in time. \\
We aim to prove the existence of weak solutions of \eqref{gpg1}, using the elliptic regularization technique of \cite{audsertil}: this technique is very flexible, introduced in \cite{apprell 4}, and developed in \cite{audsertil}, \cite{apprell 3}, \cite{apprell 1}, and \cite{apprell 2}. 
This is the strategy of the proof of the existence of a weak solution: taking a small parameter $\varepsilon>0$, we want to find a minimizer $v_\varepsilon$ on a suitable space $\mathcal{U}$ of the perturbed functional
\begin{equation*}
\mathcal{F}_\varepsilon(v):=\int_0^{+\infty} \int_\Omega \frac{e^{-\frac{t}{\varepsilon}}}{\varepsilon} \{ \varepsilon r(x) \cdot |\partial_t v|^2 + d(t) \cdot |\nabla v|^2 -2F(t,x,v) + \frac{\beta(t,x)}{2} \braket{v^2,A(t,x)v^2} \} \, dxdt,
\end{equation*}
where 
\begin{align*}
& r(x) \cdot |\partial_t v|^2:= \sum_{i=1}^k r_i(x) |\partial_t v_i|^2; \quad d(t) \cdot |\nabla v|^2:= \sum_{i=1}^k d_i(t) |\nabla v_i|^2; \\ 
&F(t,x,v(x)):= \sum_{i=1}^k \int_0^{v(x)} f_i(t,x,l) \, dl; \quad \braket{v^2, A(t,x)v^2}:=\sum_{i,j=1}^k a_{ij}(t,x)v^2_iv^2_j.
\end{align*}
This minimizer $v_\varepsilon$ solves, in a weak sense, for every $i=1, \dots, k$, the Euler-Lagrange equation:
\begin{equation*}
-\varepsilon r_i(x)\partial_{tt}v_{\varepsilon,i} +r_i(x)\partial_t v_{\varepsilon,i} -d_i(t)\Delta v_{\varepsilon,i}=f_i(t,x,v_{\varepsilon,i})-\beta(t,x)v_{\varepsilon,i}\sum_{j\ne i} a_{ij}(t,x)v^2_{\varepsilon,j} \quad \text{in} \, \Omega \times \mathbb{R}^+.
\end{equation*}
Finally, we aim to pass to the limit for $\varepsilon \to 0$, leveraging uniform estimates on $\varepsilon$ to obtain the weak solution to \eqref{gpg1}. \\
The main difficulty respect to \cite{audsertil} is that we need suitable assumption on time derivatives of the coefficients to obtain uniform estimates. \\
In this work we also study the asymptotic behavior of a weak solution of \eqref{gpg1}: considering for example the following significant case of \eqref{gpg1}
\begin{equation} \label{gpg2}
\begin{cases}
\partial_t w_{i} -\Delta w_{i}=f_i(w_{i})-\beta(t)w_{i}\sum_{j \ne i}a_{ij}w^2_{j} \quad \forall i=1,\dots,k, \\
w_{i}(0,x)=v_{0,i}(x) \quad \text{in} \,\, \Omega, \quad \text{in} \, \Omega \times \mathbb{R}^+ \quad \forall i=1, \dots,k, \\
w_{i}=g_i \quad \text{on} \,\, \partial \Omega, \quad \forall i=1,\dots,k, \\
0 \le w_{i} \le 1 \quad \forall i=1,\dots,k,
\end{cases}
\end{equation}
with the function $\beta$ such that $\lim_{t \to +\infty} \beta(t)=+\infty$, we show the segregation phenomenon for $t \to +\infty$.
To obtain this result we leverage the following estimate:
\begin{equation} \label{gpg3}
\int_T^{T+\tau} \int_{\Omega_\delta} |\nabla w_i|^2 +\beta(t) w^2_i \sum_{j \ne i} a_{ij}w^2_j \, dxdt \le C_{\delta,\tau},
\end{equation}
for every $i=1, \dots,k$, for every $T>0$, and for every $\delta>0$ and $\tau>0$; where \\ $\Omega_\delta=\set{x \in \Omega \,:\, d(x, \partial \Omega) > \delta}$, and $C_{\delta,\tau}$ is a constant dependent only on $\delta$ and $\tau$.
To obtain \eqref{gpg3} we choose suitable smooth functions to test with a weak solution of \eqref{gpg2}. \\

\section{Functional setting and main results}
Let $\Omega \subset \mathbb{R}^d$ be an open, bounded and smooth domain. \\
Let  $k \ge 2$ be a fixed integer.

\subsection{Hypotheses on the coefficients} \label{hyponcoeff} 
Let us consider $U_1,U_2,U_3: \Omega \to \mathbb{R}^+$ measurable functions that satisfy
\begin{equation*}
U_1U_3 \in L^1(\Omega); \quad \quad U_2 \in L^1(\Omega).
\end{equation*}
\begin{enumerate}
\item Let us consider $A=A(t,x):[0,+\infty) \times \Omega \to Sym_{k\times k}$, $A(t,x):=(a_{ij}(t,x))_{i,j}$ a function that satisfies:
\begin{itemize} \label{ip1}
\item $a_{ii}=0; \quad a_{ij}(t,x)=a_{ji}(t,x) \ge 0 \,\,\forall i\ne j, \quad \forall (t ,x) \in [0,+\infty)\times \Omega;$
 \item $a_{ij} \,\, \text{are continuous and differentiable in $t$};$
\item $\exists \, A_1,A_2 \ge 0 \,\, \text{such that} \,\, a_{ij}(t,x),|\partial_ta_{ij}(t,x)| \le A_1e^{A_2t}U_1(x)   \,\, \forall \,(t,x) \in [0,+\infty) \times \Omega, \,\, \\ \forall \, i,j=1\dots,k.$
\end{itemize}
\item Let us consider $f_i=f_i(t,x,s): [0,+\infty) \times \Omega \times \mathbb{R} \to \mathbb{R}$ functions that satisfy, for every $i=1,\dots,k$:
\begin{itemize} \label{ip2}
\item $f_i \,\, \text{are continuous in $s$};$
\item $f_i(\cdot,\cdot,s) \ge 0 \,\, \text{in} \,\, (-\infty,0], \quad f_i(\cdot,\cdot,s) \le 0 \,\, \text{in} \,\, [1,+\infty), \, \forall \, (t,x) \in [0,+\infty) \times \Omega;$
\item $f_i \,\, \text{are continuous and differentiable in $t$};$
\item $\exists \, H_1, H_2 \ge 0 \,\, \text{such that} \,\, |f_i(t,x,s)|,|\partial_t f_i(t,x,s)| \le H_1e^{H_2t}U_2(x) \,\, \\ \forall \, (t,x,s) \in [0,+\infty) \times \Omega \times I, \,\, \forall \, i=1,\dots,k,$
\end{itemize}
where $I$ is an open set such that $[0,1] \subset I$. \\
We are mainly interested to functions $f_i=f_i(t,x,s): [0,+\infty) \times \Omega \times [0,1] \to \mathbb{R}$, with any extension in $[0,+\infty) \times \Omega \times \mathbb{R}$ that satisfies the previous hypotheses. \\
\item Let us consider $\beta=\beta(t,x): [0,+\infty) \times \Omega \to \mathbb{R}^+$ a function that satisfies:
\begin{itemize} \label{ip5}
\item $\beta \,\, \text{is continuous and differentiable in $t$};$
\item $\exists \, C_1,C_2>0 \,\, \text{such that} \,\, \beta (t,x),|\partial_t \beta(t,x)| \le C_1 e^{C_2t}U_3(x) \,\, \forall \, (t,x) \in [0,+\infty) \times \Omega.$
\end{itemize}
\item Let us consider $d_i=d_i(t): [0,+\infty)  \to \mathbb{R}^+$ functions that satisfy, for every $i=1,\dots,k$:
\begin{itemize} \label{ip6}
\item $\exists D_1 > 0, D_2 \ge 0 \, \text{and a locally bounded function $\frac{1}{\rho}=\frac{1}{\rho(t)}>0$ in $[0,+\infty)$} \,\, \text {such that} \,\, \\ \rho(t) \le d_i(t) \le D_1e^{D_2t} \,\, \forall \, t \in [0,+\infty), \, \forall i=1,\dots,k;$
\item $d_i \,\, \text{are continuous and differentiable};$
\item $\exists \, D_3 \ge 0,  \,\, \text{such that} \,\,  |d'_i(t)| \le D_3 d_i(t) \,\, \forall \, t \in [0,+\infty),  \,\, \forall \, i=1,\dots,k.$
\end{itemize}
\item Let us consider $r_i=r_i(x):\Omega \to \mathbb{R}$ measurable functions that satisfy, for every $i=1,\dots,k$:
\begin{itemize} \label{ip7}
\item $\exists \, R_1,R_2>0 \quad \text{such that} \,\, R_1 \le r_i(x) \le R_2 \quad \forall \, x \in \Omega, \,\, \forall \, i=1,\dots,k.$
\end{itemize}
\end{enumerate}

\subsection{Notations}
Given a vector-valued function $v=(v_1, \dots , v_k)$, we write:
$$|v|^2:= \sum_{i=1}^k |v_i|^2, \quad |\partial_tv|^2:= \sum_{i=1}^k |\partial_tv_i|^2, \quad |\nabla v|^2:=\sum_{i=1}^k |\nabla v_i|^2.$$
We write, from now on, for simplicity 
$$||\cdot||:=||\cdot||_{L^1(\Omega)}.$$
We also write, considering \cref{hyponcoeff}
$$\braket{v^2, A(t,x)v^2}:=\sum_{i,j=1}^k a_{ij}(t,x)v^2_iv^2_j, \quad d(t) \cdot |\nabla v|^2:= \sum_{i=1}^k d_i(t) |\nabla v_i|^2, \quad r(x) \cdot |\partial_t v|^2:= \sum_{i=1}^k r_i(x) |\partial_t v_i|^2.$$
We define, for every $i=1, \dots,k$ the function $F_i : [0,+\infty)\times \Omega \times \mathbb{R} \to \mathbb{R}$:
\begin{equation} \label{ip3}
F_i(t,x,s)=\int_0^s f_i(t,x,l) \, dl,
\end{equation}
and we introduce:
$$F(t,x,v):=\sum_{i=1}^k F_i(t,x,v_i).$$

\subsection{Functional setting}
Let
\begin{equation} \label{ip13}
v_0=(v_{0,1}, \dots, v_{0,k}) \in H^1(\Omega)^k
\end{equation}
be a function satisfying the bounds
\begin{equation} \label{ip11}
0 \le v_0 \le 1, \quad (\text{that is} \,\,0 \le v_{0,i} \le 1 \quad \forall i=1, \dots, k),
\end{equation}
and let
\begin{equation} \label{ip12}
g=(g_{1}, \dots , g_{k}) \in H^{\frac{1}{2}}(\partial \Omega)^k
\end{equation}
be the trace of $v_0$ on $\partial \Omega$. \\
Let us define the following functional spaces:
\begin{align} \label{ip9}
&\mathcal{U}_{v_0,g}:=\set{u \in \cap_{T>0} H^1(\Omega \times (0,T))^k  \, : \, u(\cdot,0)=v_0, \,\, u(\cdot,t)=g \, \text{on} \, \partial \Omega \,\, \text{for a.e.} \, t>0}; \\ \label{ip10}
&\mathcal{U}_{v_0}:=\set{u \in \cap_{T>0} H^1(\Omega \times (0,T))^k  \, : \, u(\cdot,0)=v_0}.
\end{align}
For every $\varepsilon \in (0,\bar{\varepsilon})$ we consider the following class of functionals $\mathcal{F}_\varepsilon:\mathcal{U} \to \bar{\mathbb{R}}$:
\begin{equation*}
\mathcal{F}_\varepsilon(v):=\int_0^{+\infty} \int_\Omega \frac{e^{-\frac{t}{\varepsilon}}}{\varepsilon} \{ \varepsilon r(x) \cdot |\partial_t v|^2 + d(t) \cdot |\nabla v|^2 -2F(t,x,v) + \frac{\beta(t,x)}{2} \braket{v^2,A(t,x)v^2} \} \, dxdt,
\end{equation*}
where $\mathcal{U}$ is $\mathcal{U}_{v_0,g}$ or $\mathcal{U}_{v_0}$.

\subsection{Main results}
The main results that we want to prove in this work are the following:
\begin{theorem}[Existence theorem] \label{existence}
Under the conditions of \cref{hyponcoeff}, and \eqref{ip13}, \eqref{ip11}, \eqref{ip12}, considering \eqref{ip9}, \eqref{ip10}, there exists $w=(w_1,\dots,w_k) \in \mathcal{U}$, that is a weak solution, in the sense of $H^{-1}(\Omega \times (0,T))^k$ for all $T>0$ (and, in particular, in the sense of distribution of $\Omega \times \mathbb{R}^+$), of
\begin{equation*}
\begin{cases}
r_i(x)\partial_t w_{i} -d_i(t)\Delta w_{i}=f_i(t,x,w_{i})-\beta(t,x)w_{i}\sum_{j \ne i}a_{ij}(t,x)w^2_{j} \quad \forall i=1,\dots,k, \\
w_{i}(0,x)=v_{0,i}(x) \quad \text{in} \,\, \Omega, \quad \forall i=1, \dots,k, \\
w_{i}=g_i \quad \text{on} \,\, \partial \Omega, \quad \forall i=1,\dots,k, \\
0 \le w_{i} \le 1 \quad \forall i=1,\dots,k,
\end{cases}
\end{equation*}
if $\mathcal{U}=\mathcal{U}_{v_0,g}$, or of
\begin{equation*}
\begin{cases}
r_i(x)\partial_t w_{i} -d_i(t)\Delta w_{i}=f_i(t,x,w_{i})-\beta(t,x)w_{i}\sum_{j \ne i}a_{ij}(t,x)w^2_{j} \quad \forall i=1,\dots,k, \\
w_{i}(0,x)=v_{0,i}(x) \quad \text{in} \,\, \Omega, \quad \forall i=1, \dots,k, \\
0 \le w_{i} \le 1 \quad \forall i=1,\dots,k,
\end{cases}
\end{equation*}
if $\mathcal{U}=\mathcal{U}_{v_0}$.
\end{theorem}
\begin{theorem}[Limit segregation principle] \label{lsp}
Let us assume the conditions of \cref{hyponcoeff}. \\
Let us assume also that $D_2,H_2=0$ in \eqref{ip2} and \eqref{ip6} of \cref{hyponcoeff}. \\
Let us assume that there exists a function $b:[0,+\infty) \to \mathbb{R}^+$ such that 
\begin{equation*}
\beta(t,x) \ge b(t) \,\, \forall \, (t,x) \in [0,+\infty) \times \Omega,
\end{equation*}
with
\begin{equation*}
\lim_{t \to +\infty} b(t)=+\infty.
\end{equation*}
Let us assume that there exists $\bar{\mu}>0$ such that 
\begin{equation*}
a_{i,j}(t,x) \ge \bar{\mu} \,\, \forall \, i\ne j, \,\, \forall \, (t,x) \in [0,+\infty) \times \Omega.
\end{equation*}
Let $w=(w_1, \dots ,w_k)\in  \cap_{T>0} H^1(\Omega \times (0,T))^k$ be a weak solution, in the sense of $H^{-1}(\Omega \times (0,T))^k$ for all $T>0$, of 
\begin{equation*}
\begin{cases}
r_i(x)\partial_t w_{i} -d_i(t)\Delta w_{i}=f_i(t,x,w_{i})-\beta(t,x)w_{i}\sum_{j\ne i}a_{ij}(t,x)w^2_{j} \quad \forall i=1,\dots,k, \\
0 \le w_{i} \le 1 \quad \forall i=1,\dots,k.
\end{cases}
\end{equation*}
If there exists $\bar{v}=(\bar{v}_1,\dots \bar{v}_k) \in L^1_{loc}(\Omega)^k$ such that
\begin{equation*}
w(t) \xrightarrow[t \to +\infty]{} \bar{v} \quad \text{in} \, L^1_{loc}(\Omega),
\end{equation*}
then $\bar{v}$ is a segregated function, that is
\begin{equation*}
\bar{v}_i\bar{v}_j=0 \quad \text{a.e. in} \,\, \Omega, \quad \forall \, i \ne j.
\end{equation*}
\end{theorem}

\section{Existence of weak solutions}

To prove the existence of a weak solution of \eqref{gpg1}, we need to prove the existence, for $\varepsilon>0$ sufficiently small, of a minimizer for $\mathcal{F}_\varepsilon$. Then we want to obtain uniform estimates on $\varepsilon$ for these minimizers: so we can pass to the limit in $\varepsilon$, and find a solution for our reaction-diffusion system. \\

\begin{theorem}[Existence of minimizers] \label{eomn1}
Let us assume the conditions of \cref{hyponcoeff}, and \eqref{ip3}, \eqref{ip13}, \eqref{ip11}, \eqref{ip12}. 
Let $\mathcal{U}$ be $\mathcal{U}_{v_0,g}$ or $\mathcal{U}_{v_0}$, as in \eqref{ip9}, \eqref{ip10}. \\
Then, there exists $\bar{\varepsilon}>0$ such that, for every $\varepsilon \in (0,\bar{\varepsilon})$, there exists $v_\varepsilon=(v_{\varepsilon,1}, \dots , v_{\varepsilon,k}) \in \mathcal{U}$ minimizer for $\mathcal{F}_\varepsilon$ on $\mathcal{U}$, such that $0 \le v_{\varepsilon,i} \le 1$ a.e. in $\Omega \times \mathbb{R}^+$ for every $i=1, \dots, k$.
\end{theorem}
\begin{proof}
First we notice that $\mathcal{U} \ne \emptyset$.
In fact the function $v_0(t,x)=v_0(x)$ belongs to $\mathcal{U}$: moreover $0 \le v_0 \le 1$. \\
We note that, for every $i=1,\dots,k$
$$\sup_{s \in \mathbb{R}} F_i(t,x,s)= \sup_{s \in [0,1]} F_i(t,x,s) \le \sup_{s \in [0,1]} |F_i(t,x,s)| \le H_1e^{H_2t}U_2(x) \quad \forall (t,x) \in [0,+\infty) \times \Omega,$$
and in particular
$$-F_i(t,x,s) \ge - \sup_{s \in [0,1]} |F_i(t,x,s)| \ge -H_1e^{H_2t}U_2(x)  \quad \forall (t,x) \in [0,+\infty) \times \Omega \times \mathbb{R}.$$
Let us define
\begin{equation} \label{ip8}
\bar{\varepsilon}:=\frac{1}{2(A_2+H_2+C_2+D_2+1)},
\end{equation}
and let 
\begin{equation} \label{ne11}
\varepsilon \in (0,\bar{\varepsilon})
\end{equation} be fixed. \\
By \eqref{ip8} and \eqref{ne11}, for every $v \in \mathcal{U}$
\begin{equation} \label{ne122}
\mathcal{F}_\varepsilon(v)\ge -2kH_1\int_0^{+\infty} \frac{e^{-\frac{t}{\varepsilon}}}{\varepsilon}e^{H_2t} \int_\Omega U_2(x) \, dxdt\ge -4kH_1||U_2||=-M_1>-\infty.
\end{equation}
Since $v_0$ is an element of $\mathcal{U}$, we have that
\begin{align} \label{ne123}
\notag \mathcal{F}_\varepsilon(v_0)&\le \int_0^{+\infty} \int_\Omega \frac{e^{-\frac{t}{\varepsilon}}}{\varepsilon} \{ \sum_{i=1}^k |d_i(t)| |\nabla v_{0,i}|^2 +2\sum_{i=1}^k \sup_{s \in[0,1]} |F_i(t,x,s)|+\frac{|\beta(t,x)|}{2}\sum_{i,j=1}^k a_{ij}\} \, dxdt  \\
\notag & \le \int_0^{+\infty} \frac{e^{-\frac{t}{\varepsilon}}}{\varepsilon} (D_1e^{D_2t}||\nabla v_0||^2_{L^2(\Omega)^k}+2kH_1||U_2||e^{H_2t}+\frac{k^2C_1A_1}{2}||U_1U_3||e^{(C_2+A_2)t}) \, dt \\
& \le 2(D_1||\nabla v_0||^2_{L^2(\Omega)^k}+2kH_1||U_2|| +\frac{k^2C_1A_1}{2}||U_1U_3||)=M_2 <+\infty.
\end{align}
So the infimum exists and it is finite. \\
Let $\set{u_n}_n \subset \mathcal{U}$ be a minimizing sequence for $\mathcal{F}_\varepsilon$.
Since $\mathcal{F}_\varepsilon(u_n) \le M_2$ and \\ $\int_0^{+\infty} \int_\Omega \frac{e^{-\frac{t}{\varepsilon}}}{\varepsilon} \{-2F(t,x,u_n)\} \,dxdt \ge -M_1$, we have that
$$\int_0^{+\infty}  \frac{e^{-\frac{t}{\varepsilon}}}{\varepsilon}\int_\Omega \varepsilon r(x) \cdot |\partial_tu_n|^2 \, dxdt + \int_0^{+\infty}  \frac{e^{-\frac{t}{\varepsilon}}}{\varepsilon} \int_\Omega d(t) \cdot |\nabla u_n|^2 \, dxdt  \le M_2+M_1. $$
By supposing that 
$$0 \le u_n \le 1 \quad \forall n \in \mathbb{N},$$
we notice that the sequence $\set{\max(0,\min(u_n,1))}_n$ is a minimizing sequence: indeed 
$$\mathcal{F}_\varepsilon(\max(0,\min(u_n,1)) \le \mathcal{F}_\varepsilon(u_n) \quad \forall n \in \mathbb{N}.$$
Let $T>0$ be fixed. \\
We call $M_3:=M_2+M_1$, and $\tilde{\rho}(T):=\frac{1}{\sup_{[0,T]} \frac{1}{\rho(t)}}>0$, where $\rho$ is defined in \eqref{ip6} of \cref{hyponcoeff}. \\
Observing that
\begin{align*}
&e^{-\frac{T}{\varepsilon}}\int_0^{T}  R_1 \int_\Omega |\partial_tu_n|^2 \, dxdt \le \int_0^{T}  \frac{e^{-\frac{t}{\varepsilon}}}{\varepsilon}\int_\Omega \varepsilon r(x) \cdot |\partial_tu_n|^2 \, dxdt  \le M_3, \\
&\frac{e^{-\frac{T}{\varepsilon}}}{\varepsilon}\tilde{\rho}(T)\int_0^{T} \int_\Omega  |\nabla u_n|^2 \, dxdt \le \int_0^{T}  \frac{e^{-\frac{t}{\varepsilon}}}{\varepsilon} \int_\Omega d(t) \cdot |\nabla u_n|^2 \, dxdt \le M_3,
\end{align*}
we have that
\begin{align*}
&\int_0^{T}  \int_\Omega |\partial_tu_n|^2 \, dxdt \le \frac{M_3}{R_1} e^\frac{T}{\varepsilon}, \\
&\int_0^{T} \int_\Omega  |\nabla u_n|^2 \, dxdt \le \frac{M_3\varepsilon e^\frac{T}{\varepsilon}}{\tilde{\rho}(T)}.
\end{align*}
By the fact that $0 \le u_n \le 1$ for all $n \in \mathbb{N}$, we obtain
$$\int_0^T\int_\Omega |u_n|^2 \, dxdt \le kT|\Omega|.$$
So there exist a converging subsequence (not relabeled), and a limit $\bar{u} \in \cap_{T>0} H^1(\Omega \times (0,T))^k$ such that, for $n \to +\infty$
$$u_n \xrightharpoonup[n \to +\infty]{} \bar{u} \quad \text{in} \,\, H^1(\Omega \times (0,T))^k \quad \forall \, T>0, \,\, \text{and a.e. in} \,\, \Omega \times \mathbb{R}^+.$$
Observing that the set $\mathcal{U}$ is strongly closed and convex, $\bar{u} \in \mathcal{U}$.
By the fact that $u_n \to \bar{u}$ a.e. in $\Omega \times \mathbb{R}^+$ and $0 \le u_n \le 1$ for all $n$, we have that $0 \le \bar{u} \le 1$ in $\Omega \times \mathbb{R}^+$. \\
Considering 
$$\mathcal{F}_\varepsilon(\bar{u})=\int_0^{+\infty} \int_\Omega \frac{e^{-\frac{t}{\varepsilon}}}{\varepsilon} \{ \varepsilon r(x) \cdot |\partial_t \bar{u}|^2 + d(t) \cdot |\nabla \bar{u}|^2 -2F(t,x,\bar{u}) + \frac{\beta(t,x)}{2} \braket{\bar{u}^2,A(t,x)\bar{u}^2} \} \, dxdt,$$
we get, by the Fatou Lemma:
\begin{equation*}
\int_0^{+\infty} \int_\Omega \frac{e^{-\frac{t}{\varepsilon}}}{\varepsilon} \frac{\beta(t,x)}{2} \braket{\bar{u}^2,A(t,x)\bar{u}^2} \, dxdt \le \liminf_{n \to +\infty} \int_0^{+\infty} \int_\Omega \frac{e^{-\frac{t}{\varepsilon}}}{\varepsilon} \frac{\beta(t,x)}{2} \braket{u^2_n,A(t,x)u^2_n} \, dxdt,
\end{equation*}
and by the dominate convergence:
\begin{equation*}
\int_0^{+\infty} \int_\Omega \frac{e^{-\frac{t}{\varepsilon}}}{\varepsilon} \{ -2F(t,x,\bar{u}) \} \, dxdt=\lim_{n\to +\infty} \int_0^{+\infty} \int_\Omega  \frac{e^{-\frac{t}{\varepsilon}}}{\varepsilon} \{ -2F(t,x,u_n) \} \, dxdt.
\end{equation*}
By the convexity and the weak convergence we have that, for all $T>0$
\begin{align*}
&\int_0^{T} \int_\Omega \frac{e^{-\frac{t}{\varepsilon}}}{\varepsilon} \{ \varepsilon r(x) \cdot |\partial_t \bar{u}|^2 +d(t) \cdot |\nabla \bar{u}|^2 \} \, dxdt \\
& \le \liminf_{n \to +\infty} \int_0^{T} \int_\Omega \frac{e^{-\frac{t}{\varepsilon}}}{\varepsilon} \{ \varepsilon r(x) \cdot |\partial_t u_n|^2 + d(t) \cdot |\nabla u_n|^2 \} \, dxdt \\
& \le  \liminf_{n \to +\infty} \int_0^{+\infty} \int_\Omega \frac{e^{-\frac{t}{\varepsilon}}}{\varepsilon} \{ \varepsilon r(x) \cdot |\partial_t u_n|^2 +d(t) \cdot |\nabla u_n|^2 \} \, dxdt,
\end{align*}
and since it holds for every $T>0$, we obtain
\begin{align*}
&\int_0^{+\infty} \int_\Omega \frac{e^{-\frac{t}{\varepsilon}}}{\varepsilon} \{ \varepsilon r(x) \cdot |\partial_t \bar{u}|^2 +d(t) \cdot |\nabla \bar{u}|^2 \} \, dxdt \\
&\le \liminf_{n \to +\infty} \int_0^{+\infty} \int_\Omega \frac{e^{-\frac{t}{\varepsilon}}}{\varepsilon} \{ \varepsilon r(x) \cdot |\partial_t u_n|^2 +d(t) \cdot  |\nabla u_n|^2 \} \, dxdt.
\end{align*}
So
$$\mathcal{F}_\varepsilon(\bar{u}) \le \liminf_{n \to+\infty} \mathcal{F}_\varepsilon(u_n)=\inf_\mathcal{U} \mathcal{F}_\varepsilon.$$
Hence $\bar{u}=v_\varepsilon$, getting the thesis.
\end{proof}
\begin{corollary} \label{appeq}
Let us assume the conditions of \cref{hyponcoeff}, and \eqref{ip3}, \eqref{ip13}, \eqref{ip11}, \eqref{ip12}. 
Let $\mathcal{U}$ be $\mathcal{U}_{v_0,g}$ or $\mathcal{U}_{v_0}$ like in \eqref{ip9}, \eqref{ip10}. 
Let $\bar{\varepsilon}$ be as in \eqref{ip8}. \\
For every $\varepsilon \in (0,\bar{\varepsilon})$, the minimizer $v_\varepsilon$ is a weak solution, in the sense of $H^{-1}(\Omega \times (0,T))^k$ for all $T>0$ (and, in particular, in the sense of distribution of $\Omega \times \mathbb{R}^+$), of:
\begin{equation} \label{equaappr1}
\begin{cases}
-\varepsilon r_i(x)\partial_{tt}v_{\varepsilon,i} +r_i(x)\partial_t v_{\varepsilon,i} -d_i(t)\Delta v_{\varepsilon,i}=f_i(t,x,v_{\varepsilon,i})-\beta(t,x)v_{\varepsilon,i}\sum_{j\ne i} a_{ij}(t,x)v^2_{\varepsilon,j} \,\, \forall i=1,\dots,k, \\
v_{\varepsilon,i}(0,x)=v_{0,i}(x) \quad \text{in} \,\, \Omega, \quad \forall i=1, \dots,k, \\
v_{\varepsilon,i}=g_i \quad \text{on} \,\, \partial \Omega, \quad \forall i=1,\dots,k, \\
0 \le v_{\varepsilon,i} \le 1 \quad \forall i=1,\dots,k,
\end{cases}
\end{equation}
if $\mathcal{U}=\mathcal{U}_{v_0,g}$, or of
\begin{equation} \label{equaappr2}
\begin{cases}
-\varepsilon r_i(x)\partial_{tt}v_{\varepsilon,i} +r_i(x)\partial_t v_{\varepsilon,i} -d_i(t)\Delta v_{\varepsilon,i}=f_i(t,x,v_{\varepsilon,i})-\beta(t,x)v_{\varepsilon,i}\sum_{j \ne i} a_{ij}(t,x)v^2_{\varepsilon,j} \,\, \forall i=1,\dots,k, \\
v_{\varepsilon,i}(0,x)=v_{0,i}(x) \quad \text{in} \,\, \Omega, \quad \forall i=1, \dots,k, \\
0 \le v_{\varepsilon,i} \le 1 \quad \forall i=1,\dots,k,
\end{cases}
\end{equation}
if $\mathcal{U}=\mathcal{U}_{v_0}$.
\end{corollary}
\begin{proof}
Let $\varepsilon \in (0,\bar{\varepsilon})$ be fixed. \\
For every $i=1,\dots k$ fixed, we consider $\varphi_i \in C^\infty_c(\Omega \times \mathbb{R}^+)$, and $\phi=(\phi_1,\dots ,\phi_k) \in C^\infty_c(\Omega \times \mathbb{R}^+)^k$ such that $\phi_i=\varphi_i$, and $\phi_j=0$ for all $j \ne i$. \\
Since $v_\varepsilon$ is a minimizer of $\mathcal{F}_\varepsilon$ in $\mathcal{U}$ we get, by dominate convergence, that
\begin{align*}
0 &= \lim_{h \to 0} \frac{\mathcal{F}_\varepsilon(v_\varepsilon +h\phi)-\mathcal{F}_\varepsilon(v_\varepsilon)}{h} \\
&=\lim_{h \to 0} \frac{1}{h} \int_0^{+\infty} \int_\Omega \frac{e^{-\frac{t}{\varepsilon}}}{\varepsilon} \{ 2h\varepsilon r_i(x)\partial_t v_{\varepsilon,i} \partial_t \varphi_i +h^2\varepsilon r_i(x)|\partial_t\varphi_i|^2+ 2hd_i(t)\nabla v_{\varepsilon,i}\cdot \nabla \varphi_i +h^2d_i(t)|\nabla \varphi_i|^2  \\
&\qquad \qquad \qquad \qquad \qquad -2(F_i(t,x,v_{\varepsilon,i}+h\varphi_i)-F_i(t,x,v_{\varepsilon,i})) \\
&\qquad \qquad \qquad \qquad \qquad +2h\beta(t,x)v_{\varepsilon,i}\sum_{j \ne i} a_{ij}(t,x)v^2_{\varepsilon,j}\varphi_i+h^2\beta(t,x)\sum_{j \ne i}a_{ij}(t,x)v^2_{\varepsilon,j}\varphi^2_i \} \, dxdt \\
&=2\int_0^{+\infty} \int_\Omega \frac{e^{-\frac{t}{\varepsilon}}}{\varepsilon} \{ \varepsilon  r_i(x)\partial_t v_{\varepsilon,i} \partial_t \varphi_i+d_i(t) \nabla v_{\varepsilon,i}\cdot \nabla \varphi_i -f_i(t,x,v_{\varepsilon,i})\varphi_i \\
& \qquad \qquad \qquad \qquad +\beta(t,x)v_{\varepsilon,i}\sum_{j \ne i} a_{ij}(t,x)v^2_{\varepsilon,j}\varphi_i \} \,dxdt.
\end{align*}
Choosing $\varphi_i=\varepsilon e^\frac{t}{\varepsilon}\varphi$, where $\varphi \in C^\infty_c(\Omega \times \mathbb{R}^+)$ is arbitrary, we have that
\begin{align*}
&\int_0^{+\infty} \int_\Omega \varepsilon r_i(x)\partial_t v_{\varepsilon,i} \partial_t \varphi+ r_i(x)\partial_t v_{\varepsilon,i}\varphi+d_i(t)\nabla v_{\varepsilon,i}\cdot \nabla \varphi -f_i(t,x,v_{\varepsilon,i})\varphi \\
&\qquad \qquad \qquad \qquad \qquad \qquad \qquad \qquad +\beta(t,x)v_{\varepsilon,i}\sum_{j\ne i} a_{ij}(t,x)v^2_{\varepsilon,j}\varphi \, dxdt=0.
\end{align*}
For every $T>0$ fixed, by the density of $C^\infty_c(\Omega \times (0,T))$ in $H^1_0(\Omega \times (0,T))$ with the $H^1$-norm, we have that the previous equation is also true for every $\varphi \in H^1_0(\Omega \times (0,T))$, for every $T>0$.
Thus $v_{\varepsilon,i}$ is a weak solution, getting the thesis.
\end{proof}
Now we want uniform estimates to pass to the limit for $\varepsilon \to 0$.
\begin{corollary}
Let us assume the conditions of \cref{hyponcoeff}, and \eqref{ip3}, \eqref{ip13}, \eqref{ip11}, \eqref{ip12}. 
Let $\bar{\varepsilon}$ be as in \eqref{ip8}. \\
For every $\varepsilon \in (0,\bar{\varepsilon})$, let $v_\varepsilon$ be a minimizer of $\mathcal{F}_\varepsilon$.
Then there exists a constant $C>0$ independent of $\varepsilon$ such that
\begin{equation*}
|\mathcal{F}_\varepsilon(v_\varepsilon)| \le C.
\end{equation*}
\end{corollary}
\begin{proof}
By \eqref{ne122} and \eqref{ne123}, we know that
$$\mathcal{F}_\varepsilon(v) \ge -M_1 \quad \forall v \in \mathcal{U}, \quad \quad  \mathcal{F}_\varepsilon(v_\varepsilon) \le \mathcal{F}_\varepsilon(v_0) \le M_2.$$
Taking $C=M_1+M_2$ we get the thesis.
\end{proof}
We define the following quantities, for a.e. $t \in \mathbb{R}^+$:
\begin{align*}
&R_\varepsilon(t):=\int_\Omega d(t) \cdot |\nabla v_\varepsilon(t)|^2 -2F(t,x,v_\varepsilon(t)) +\frac{\beta(t,x)}{2} \braket{v^2_\varepsilon(t),A(t,x)v^2_\varepsilon(t)} \, dx; \\
&I_\varepsilon(t):=\varepsilon \int_\Omega r(x) \cdot |\partial_t v_\varepsilon(t)|^2 \, dx,
\end{align*}
thus
\begin{equation*}
\mathcal{F}_\varepsilon(v_\varepsilon)=\int_0^{+\infty} \frac{e^{-\frac{t}{\varepsilon}}}{\varepsilon} \{I_\varepsilon(t)+R_\varepsilon(t) \} \, dt.
\end{equation*}
Moreover, we define the energy
\begin{equation*}
E_\varepsilon(t):=e^\frac{t}{\varepsilon}\int_t^{+\infty} \frac{e^{-\frac{\tau}{\varepsilon}}}{\varepsilon} \{I_\varepsilon(\tau)+R_\varepsilon(\tau) \} \, d\tau,
\end{equation*}
and we have that
\begin{equation} \label{enidn1}
\begin{cases}
E'_\varepsilon(t)=\frac{1}{\varepsilon}(-I_\varepsilon(t)-R_\varepsilon(t)+E_\varepsilon(t)) \\
E_\varepsilon(0)=\mathcal{F}_\varepsilon(v_\varepsilon).
\end{cases}
\end{equation}

\begin{theorem}[Uniform estimates] \label{uniestimnn1}
Let us suppose the conditions of \cref{hyponcoeff}, and \eqref{ip3}, \eqref{ip13}, \eqref{ip11}, \eqref{ip12}. Let $\bar{\varepsilon}$ be as in \eqref{ip8}. Then, for every $\varepsilon \in (0,\bar{\varepsilon})$ there exists a function $\bar{v}_\varepsilon\in \mathcal U$, which is a weak solution of \eqref{equaappr1} or \eqref{equaappr2} and is such that, for every $T>0$, there exists a constant $\check{C}(T)$ that is independent of $\rho(t)$ and $\varepsilon$ such that:
\begin{align*}
&\int_0^T \int_\Omega |\partial_t \bar{v}_\varepsilon|^2 \, dxdt \le \check{C}(T); \\ 
&\int_0^T \int_\Omega |\nabla \bar{v}_\varepsilon|^2 \, dxdt \le \frac{\check{C}(T)}{(1-\frac{1}{e^{-D_2T}+1})\tilde{\rho}(T)}; \\ 
&\int_0^T \int_\Omega |\bar{v}_\varepsilon|^2 \, dxdt \le \check{C}(T). 
\end{align*}
where, for every $T>0$, we set 
\begin{equation}\label{e:definition-rho-tilde}
\tilde{\rho}(T):=\inf_{t\in[0,T]} \rho(t)>0.
\end{equation}
\end{theorem}
\begin{proof}
First, in Step 1, we prove the theorem under an additional hypothesis on the growth of the coefficients $d_i(t)$. Then, in Step 2, we use an approximation argument to prove the theorem in the general case.\medskip 

\noindent {\bf Step 1:}
First we suppose that there exist $\mu>0$ and $D_4>0$ such that
\begin{equation}\label{e:step1-hypo-ag}
\rho(t)+\mu \le d_i(t) \le D_4\quad\text{for every}\quad i=1,\dots,k. 
\end{equation}
Then, for every $T>0$, there is a constant $\bar{C}(T)$, that does not depend on $D_4$ and $\mu$, such that:
\begin{align}
&\int_0^T \int_\Omega |\partial_t v_\varepsilon|^2 \, dxdt \le \bar{C}(T);\label{e:step1-conclusion-1} \\
&\int_0^T \int_\Omega |\nabla v_\varepsilon|^2 \, dxdt \le \frac{\bar{C}(T)}{\tilde{\rho}(T)+\mu};\label{e:step1-conclusion-2} \\
&\int_0^T \int_\Omega |v_\varepsilon|^2 \, dxdt \le \bar{C}(T).\label{e:step1-conclusion-3}
\end{align}
In order to prove Step 1, we will use an inner variation argument as in \cite{audsertil}. \\
Let $\eta \in C^\infty_c(\mathbb{R}^+)$, and let
$$\xi(t):=\int_0^t \eta(\tau) d\tau, \quad \quad \phi_\delta(t):=t-\delta \xi(t),$$
with $|\delta|$ small enough such that $\phi'>0$ in $\mathbb{R}^+$. We observe that $\phi_\delta(0)=0$. \\
So we can take the inverse $\psi_\delta:=\phi^{-1}_\delta$ that satisfies, for every $\tau \in \mathbb{R}^+$:
$$\psi_\delta(\tau)=\tau+\delta \xi(\psi_\delta(\tau)).$$
Let us introduce the competitor $z_\delta(t):=v_\varepsilon(\phi_\delta(t))$. Hence,
\begin{align*}
\mathcal{F}_\varepsilon(z_\delta)&=\int_0^{+\infty} \int_\Omega \frac{e^{-\frac{t}{\varepsilon}}}{\varepsilon} \{ \varepsilon r(x) \cdot|\partial_tv_\varepsilon(\phi_\delta(t))|^2|\phi'_\delta(t)|^2 +d(t) \cdot |\nabla v_\varepsilon(\phi_\delta(t))|^2 \\
&\qquad \qquad \qquad \quad-2F(t,x,v_\varepsilon(\phi_\delta(t))) +\frac{\beta(t,x)}{2}\braket{v^2_\varepsilon(\phi_\delta(t)),A(t,x)v^2_\varepsilon(\phi_\delta(t))} \} \, dxdt \\
&=\int_0^{+\infty} \int_\Omega \frac{e^{-\frac{\psi_\delta(\tau)}{\varepsilon}}}{\varepsilon}\psi'_\delta(\tau) \{ \varepsilon r(x) \cdot |\partial_tv_\varepsilon(\tau)|^2|\phi'_\delta(\psi_\delta(\tau))|^2 +d(\psi_\delta(\tau)) \cdot |\nabla v_\varepsilon(\tau)|^2 \\
& \qquad \qquad \qquad \qquad \qquad \quad -2F(\psi_\delta(\tau),x,v_\varepsilon(\tau)) \\
& \qquad \qquad \qquad \qquad \qquad \quad +\frac{\beta(\psi_\delta(\tau),x)}{2}\braket{v^2_\varepsilon(\tau),A(\psi_\delta(\tau),x)v^2_\varepsilon(\tau)} \} \, dxd\tau.
\end{align*}
This is finite for all $|\delta| \le \min \{\frac{1}{2||\eta||_{L^\infty}},1\}$: defining
\begin{align*}
P_\delta(\tau):=&\varepsilon r(x) \cdot |\partial_tv_\varepsilon(\tau)|^2|\phi'_\delta(\psi_\delta(\tau))|^2 +d(\psi_\delta(\tau)) \cdot |\nabla v_\varepsilon(\tau)|^2 -2F(\psi_\delta(\tau),x,v_\varepsilon(\tau)) \\
&+\frac{\beta(\psi_\delta(\tau),x)}{2}\braket{v^2_\varepsilon(\tau),A(\psi_\delta(\tau),x)v^2_\varepsilon(\tau)},
\end{align*}
we get
\begin{align*}
\mathcal{F}_\varepsilon(z_\delta)&=\int_0^{+\infty} \int_\Omega \frac{e^{-\frac{\psi_\delta(\tau)}{\varepsilon}}}{\varepsilon}\psi'_\delta(\tau) P_\delta(\tau) \, dxd\tau \\
&\le \int_0^{+\infty} \int_\Omega 2\frac{e^{-\frac{\tau}{\varepsilon}}}{\varepsilon}e^\frac{ ||\xi||_{L^\infty}}{\varepsilon} \{4\varepsilon r(x) \cdot |\partial_t v_\varepsilon(\tau)|^2 +\frac{D_4}{\mu}d(\tau) \cdot |\nabla v_\varepsilon(\tau)|^2 \} \,dxd\tau \\
&+4kH_1e^{\frac{||\xi||_{L^\infty}}{\varepsilon}+H_2||\xi||_{L^\infty}}||U_2||\int_0^{+\infty} \frac{e^\frac{-(1-\varepsilon H_2)\tau}{\varepsilon}}{\varepsilon} \, d\tau \\
& +2k^2C_1A_1 e^{\frac{||\xi||_{L^\infty}}{\varepsilon}+(C_2+A_2)||\xi||_{L^\infty}}||U_3U_1|| \int_0^{+\infty} \frac{e^\frac{-(1-\varepsilon (C_2+A_2))\tau}{\varepsilon}}{\varepsilon} \, d\tau  \\
& \le 2e^\frac{||\xi||_{L^\infty}}{\varepsilon}(4+\frac{D_4}{\mu})\mathcal{F}_\varepsilon(v_\varepsilon)+2e^\frac{||\xi||_{L^\infty}}{\varepsilon}(4+\frac{D_4}{\mu})2kH_1||U_2||2 +8kH_1e^{\frac{||\xi||_{L^\infty}}{\varepsilon}+H_2||\xi||_{L^\infty}}||U_2|| \\
&+4k^2C_1A_1 e^{\frac{||\xi||_{L^\infty}}{\varepsilon}+(C_2+A_2)||\xi||_{L^\infty}}||U_3U_1|| <+\infty.
\end{align*}
Hence, for all $\varepsilon \in (0,\bar{\varepsilon})$, and for all $|\delta| \le \min \{\frac{1}{2||\eta||_{L^\infty}},1\}$, we get
$$\frac{e^{-\frac{\psi_\delta(\tau)}{\varepsilon}}}{\varepsilon}\psi'_\delta(\tau) P_\delta(\tau) \in L^1(\Omega \times \mathbb{R}^+).$$
Moreover, we have shown that there is an $L^1$ function $\delta$-independent that controls \\ $|\frac{e^{-\frac{\psi_\delta(\tau)}{\varepsilon}}}{\varepsilon}\psi'_\delta(\tau) P_\delta(\tau)|$. \\
Now we want to compute the derivative with respect to $\delta$, and using the minimality of $v_\varepsilon$, i.e.
$$\frac{d}{d\delta} \mathcal{F}_\varepsilon(z_\delta)|_{\delta=0}=0.$$
First we compute
\begin{align*}
&\frac{d}{d\delta}\biggl( \frac{e^{-\frac{\psi_\delta(\tau)}{\varepsilon}}}{\varepsilon}\psi'_\delta(\tau) P_\delta(\tau)\biggr)=-\frac{1}{\varepsilon} \frac{d}{d\delta}\psi_\delta(\tau)\frac{e^{-\frac{\psi_\delta(\tau)}{\varepsilon}}}{\varepsilon}\psi'_\delta(\tau) P_\delta(\tau)+\bigl(\frac{d}{d\delta}\psi'_\delta(\tau)\bigr)\frac{e^{-\frac{\psi_\delta(\tau)}{\varepsilon}}}{\varepsilon}P_\delta(\tau) \\
&+\frac{e^{-\frac{\psi_\delta(\tau)}{\varepsilon}}}{\varepsilon}\psi'_\delta(\tau) \{ \varepsilon r(x) \cdot |\partial_tv_\varepsilon(\tau)|^2(2\phi'_\delta(\psi_\delta(\tau))(\frac{d}{d\delta}\phi'_\delta(\psi_\delta(\tau))))+d'(\psi_\delta(\tau)) \cdot |\nabla v_\varepsilon(\tau)|^2 \frac{d}{d\delta}\psi_\delta(\tau) \\
& \qquad \qquad \qquad \quad -2\partial_tF(\psi_\delta(\tau),x,v_\varepsilon(\tau))\frac{d}{d\delta}\psi_\delta(\tau) +\frac{\partial_t\beta(\psi_\delta(\tau),x)}{2}\frac{d}{d\delta}\psi_\delta(\tau)\braket{v^2_\varepsilon(\tau),A(\psi_\delta(\tau),x)v^2_\varepsilon(\tau)}  \\
& \qquad \qquad \qquad \quad +\frac{\beta(\psi_\delta(\tau),x)}{2}(\sum_{i,j=1}^k \partial_t a_{ij}(\psi_\delta(\tau),x) v^2_{\varepsilon,i}(\tau)v^2_{\varepsilon,j}(\tau)) \frac{d}{d\delta}\psi_\delta(\tau)\}.
\end{align*}
Similarly as above, using $|\delta| \le \min \{\frac{1}{2||\eta||_{L^\infty}},1\}$, we get
$$\frac{d}{d\delta}\biggl(\frac{e^{-\frac{\psi_\delta(\tau)}{\varepsilon}}}{\varepsilon}\psi'_\delta(\tau) P_\delta(\tau)\biggr) \in L^1(\Omega \times \mathbb{R}^+).$$
Moreover, as above, there is an $L^1$ function $\delta$-independent that controls $|\frac{d}{d\delta}\bigl(\frac{e^{-\frac{\psi_\delta(\tau)}{\varepsilon}}}{\varepsilon}\psi'_\delta(\tau) P_\delta(\tau)\bigr)|$. \\
Then, we calculate
\begin{align*}
0&=\frac{d}{d\delta} \mathcal{F}_\varepsilon(z_\delta)|_{\delta=0}=\int_0^{+\infty} \int_\Omega \frac{d}{d\delta}(\frac{e^{-\frac{\psi_\delta(\tau)}{\varepsilon}}}{\varepsilon}\psi'_\delta(\tau) P_\delta(\tau))|_{\delta=0} \, dxd\tau \\
&=\int_0^{+\infty} \int_\Omega \frac{e^{-\frac{\tau}{\varepsilon}}}{\varepsilon} (\xi'(\tau)-\frac{\xi(\tau)}{\varepsilon})\{\varepsilon r(x) \cdot |\partial_tv_\varepsilon(\tau)|^2+d(\tau) \cdot |\nabla v_\varepsilon(\tau)|^2 -2F(\tau,x,v_\varepsilon(\tau)) \\ 
& \qquad \qquad \qquad \qquad \qquad \qquad \qquad +\frac{\beta(\tau,x)}{2}\braket{v^2_\varepsilon(\tau),A(\tau,x)v^2_\varepsilon(\tau)} \} \, dxd\tau  \\
&+ \int_0^{+\infty} \int_\Omega \frac{e^{-\frac{\tau}{\varepsilon}}}{\varepsilon} \{ -2\xi'(\tau) \varepsilon r(x) \cdot |\partial_t v_\varepsilon(\tau)|^2+\xi(\tau)d'(\tau)\cdot |\nabla v_\varepsilon(\tau)|^2-2\xi(\tau)\partial_tF(\tau,x,v_\varepsilon(\tau)) \\
&\qquad \qquad \qquad \quad + \xi(\tau)\frac{\partial_t \beta(\tau,x)}{2} \braket{v^2_\varepsilon(\tau),A(\tau,x)v^2_\varepsilon(\tau)} \\
&\qquad \qquad \qquad \quad +\xi(\tau)\frac{\beta(\tau,x)}{2} \sum_{i,j=1}^k\partial_t a_{ij}(\tau,x) v^2_{\varepsilon,i}(\tau)v^2_{\varepsilon,j}(\tau) \} \, dxd\tau.
\end{align*}
Now we choose, for a.e. fixed $t>0$, the function
\begin{equation*}
\xi(\tau)=\begin{cases}
0 \quad \text{if} \,\, \tau \le t; \\
\varepsilon \frac{e^\frac{t}{\varepsilon}(\tau-t)}{\lambda} \quad \text{if} \,\, t<\tau<t+\lambda; \\
\varepsilon e^\frac{t}{\varepsilon} \quad \text{if} \,\, \tau \ge t+\lambda.
\end{cases}
\end{equation*}
Defining:
\begin{align*}
Q_\varepsilon(t)&:=e^\frac{t}{\varepsilon}\int_t^{+\infty} \int_\Omega \frac{e^{-\frac{\tau}{\varepsilon}}}{\varepsilon}\{d'(\tau) \cdot |\nabla v_\varepsilon(\tau)|^2-2\partial_t F(\tau,x,v_\varepsilon(\tau))+ \frac{\partial_t\beta(\tau,x)}{2}\braket{v^2_\varepsilon(\tau),A(\tau,x)v^2_\varepsilon(\tau)}\\ 
&\qquad \qquad \qquad \qquad \quad +\frac{\beta(\tau,x)}{2} \sum_{i,j=1}^k \partial_t a_{ij}(\tau,x)v^2_{\varepsilon,i}(\tau)v^2_{\varepsilon,j}(\tau) \} \,dxd\tau,
\end{align*}
for $\lambda \to 0^+$ we have that
\begin{align*}
0&=\int_\Omega \varepsilon r(x) \cdot |\partial_tv_\varepsilon(t)|^2+d(t) \cdot |\nabla v_\varepsilon(t)|^2 -2F(t,x,v_\varepsilon(t)) +\frac{\beta(t,x)}{2}\braket{v^2_\varepsilon(t),A(t,x)v^2_\varepsilon(t)} \, dx \\
&-2\int_\Omega \varepsilon r(x) \cdot |\partial_t v_\varepsilon(t)|^2 \, dx-e^\frac{t}{\varepsilon}\int_t^{+\infty} \int_\Omega \frac{e^{-\frac{\tau}{\varepsilon}}}{\varepsilon} \{\varepsilon r(x) \cdot |\partial_tv_\varepsilon(\tau)|^2+d(\tau) \cdot |\nabla v_\varepsilon(\tau)|^2 -2F(\tau,x,v_\varepsilon(\tau))  \\
&\qquad \qquad \qquad \qquad \qquad \qquad \qquad \qquad \qquad \qquad +\frac{\beta(\tau,x)}{2}\braket{v^2_\varepsilon(\tau),A(\tau,x)v^2_\varepsilon(\tau)} \} \, dxd\tau \\
&+\varepsilon e^\frac{t}{\varepsilon}\int_t^{+\infty} \int_\Omega \frac{e^{-\frac{\tau}{\varepsilon}}}{\varepsilon}\{d'(\tau) \cdot |\nabla v_\varepsilon(\tau)|^2-2\partial_t F(\tau,x,v_\varepsilon(\tau))+ \frac{\partial_t\beta(\tau,x)}{2}\braket{v^2_\varepsilon(\tau),A(\tau,x)v^2_\varepsilon(\tau)} \\
&\qquad \qquad \qquad \qquad \quad +\frac{\beta(\tau,x)}{2} \sum_{i,j=1}^k \partial_t a_{ij}(\tau,x)v^2_{\varepsilon,i}(\tau)v^2_{\varepsilon,j}(\tau) \} \,dxd\tau \\
&=I_\varepsilon(t)+R_\varepsilon(t)-2I_\varepsilon(t)-E_\varepsilon(t)+\varepsilon Q_\varepsilon(t),
\end{align*}
that is
\begin{equation} \label{nien1}
E_\varepsilon(t)=R_\varepsilon(t)-I_\varepsilon(t)+\varepsilon Q_\varepsilon(t).
\end{equation}
Recalling \eqref{enidn1}, we have
$$E'_\varepsilon(t)=\frac{1}{\varepsilon}(-2I_\varepsilon(t)+\varepsilon Q_\varepsilon(t))=-\frac{2}{\varepsilon}I_\varepsilon(t)+Q_\varepsilon(t).$$
Thanks to the conditions of \cref{hyponcoeff},
\begin{align*}
|Q_\varepsilon(t)|&\le e^\frac{t}{\varepsilon}\int_t^{+\infty} \int_\Omega \frac{e^{-\frac{\tau}{\varepsilon}}}{\varepsilon}\{|d'(\tau)| \cdot |\nabla v_\varepsilon(\tau)|^2+2|\partial_t F(\tau,x,v_\varepsilon(\tau))|+ \frac{|\partial_t\beta(\tau,x)|}{2}\braket{v^2_\varepsilon(\tau),A(\tau,x)v^2_\varepsilon(\tau)}\\ 
&\qquad \qquad \qquad \qquad \quad +\frac{\beta(\tau,x)}{2} \sum_{i,j=1}^k |\partial_t a_{ij}(\tau,x)|v^2_{\varepsilon,i}(\tau)v^2_{\varepsilon,j}(\tau) \} \,dxd\tau  \\
& \le D_3E_\varepsilon(t)+D_3e^\frac{t}{\varepsilon} \int_t^{+\infty}\int_\Omega \frac{e^{-\frac{\tau}{\varepsilon}}}{\varepsilon} 2|F(\tau,x,v_\varepsilon(\tau))|  \, dxd\tau \\
&+e^\frac{t}{\varepsilon}\int_t^{+\infty} \int_\Omega \frac{e^{-\frac{\tau}{\varepsilon}}}{\varepsilon}2|\partial_t F(\tau,x,v_\varepsilon(\tau))| \, dxd\tau \\
&+ e^\frac{t}{\varepsilon}\int_t^{+\infty} \int_\Omega \frac{e^{-\frac{\tau}{\varepsilon}}}{\varepsilon}\{\frac{|\partial_t\beta(\tau,x)|}{2}\braket{v^2_\varepsilon(\tau),A(\tau,x)v^2_\varepsilon(\tau)} \\
& \qquad \qquad \qquad \qquad+\frac{\beta(\tau,x)}{2} \sum_{i,j=1}^k |\partial_t a_{ij}(\tau,x)|v^2_{\varepsilon,i}(\tau)v^2_{\varepsilon,j}(\tau) \} \,dxd\tau  \\
& \le D_3E_\varepsilon(t) +4(D_3+1)kH_1||U_2||e^{H_2t}+2k^2C_1A_1||U_3U_1||e^{(C_2+A_2)t}.
\end{align*}
Defining
\begin{align*}
&\tilde{M}_1=\tilde{M}_1(k,A_1,C_1,H_1,D_3,U_1,U_2,U_3):= 4(D_3+1)kH_1||U_2||+2k^2C_1A_1||U_3U_1||; \\
& \tilde{M}_2=\tilde{M}_2(A_2,C_2,H_2):= A_2+C_2+H_2+1,
\end{align*}
we obtain that
$$|Q_\varepsilon(t)| \le D_3 E_\varepsilon(t) +\tilde{M}_1e^{\tilde{M}_2t}.$$
Thus $E_\varepsilon$ solves
\begin{equation*}
\begin{cases}
E'_\varepsilon(t) \le -\frac{2}{\varepsilon}I_\varepsilon(t)+\tilde{M}_1e^{\tilde{M}_2t}+D_3E_\varepsilon(t); \\
E_\varepsilon(0)=\mathcal{F}_\varepsilon(v_\varepsilon) \le C.
\end{cases}
\end{equation*}
Observing that
\begin{align*}
E_\varepsilon(t) &\ge -2e^\frac{t}{\varepsilon} \int_t^{+\infty} \int_\Omega \frac{e^{-\frac{\tau}{\varepsilon}}}{\varepsilon} |F(\tau,x,v_\varepsilon(\tau))| \, dxd\tau  \\
& \ge -4kH_1||U_2||e^{H_2t} \ge -\tilde{M}_1e^{\tilde{M}_2t},
\end{align*}
we have, by Gronwall inequality
\begin{equation} \label{groneid1}
-\tilde{M}_1e^{\tilde{M}_2T} \le E_\varepsilon(T) \le e^{D_3T}\biggl(E_\varepsilon(0)+\int_0^T e^{-D_3t} \tilde{M}_1e^{\tilde{M}_2t} \, dt -\frac{2}{\varepsilon}\int_0^T e^{-D_3t}I_\varepsilon(t) \, dt\biggr).
\end{equation}
Hence
$$-\tilde{M}_1e^{\tilde{M}_2T} \le Ce^{D_3T}+\frac{\tilde{M}_1}{\tilde{M}_2}e^{(\tilde{M}_2+D_3)T}-\frac{2}{\varepsilon}\int_0^T I_\varepsilon(t) \,dt,$$
which implies
$$\int_0^T I_\varepsilon(t) \,dt \le \frac{C+\tilde{M_1}+\frac{\tilde{M}_1}{\tilde{M}_2}}{2}e^{(\tilde{M}_2+D_3)T}\varepsilon=\tilde{c}(T)\varepsilon,$$
leading to
$$\int_0^T \int_\Omega |\partial_t v_\varepsilon|^2 \, dxdt \le \frac{\tilde{c}(T)}{R_1}.$$
Second, we want to estimate $\int_0^T \int_\Omega |\nabla v_\varepsilon|^2 \, dxdt$. \\
We know, by \eqref{groneid1}, that
$$E_\varepsilon(t) \le e^{D_3t}C+\frac{\tilde{M}_1}{\tilde{M}_2}e^{(\tilde{M}_2+D_3)t}\le (C+\frac{\tilde{M}_1}{\tilde{M}_2})e^{(\tilde{M}_2+D_3)t}.$$
By \eqref{nien1} we have:
$$R_\varepsilon(t)\le E_\varepsilon(t) +I_\varepsilon(t)+|Q_\varepsilon(t)|\le (1+D_3)E_\varepsilon(t)+\tilde{M}_1e^{\tilde{M}_2t}+I_\varepsilon(t).$$
Hence
\begin{align*}
\int_\Omega d(t) \cdot |\nabla v_\varepsilon(t)|^2 +\frac{\beta(t,x)}{2} \braket{v^2_\varepsilon(t),A(t,x)v^2_\varepsilon(t)} \, dx &= R_\varepsilon(t)+ \int_\Omega 2F(t,x,v_\varepsilon(t)) \, dx  \\
& \le R_\varepsilon(t)+ 2kH_1||U_2||e^{H_2t}  \\
& \le (1+D_3)E_\varepsilon(t)+2\tilde{M}_1e^{\tilde{M}_2t}+I_\varepsilon(t),
\end{align*}
implying, for every $T>0$, that
\begin{align*}
&\int_0^T \int_\Omega d(t) \cdot |\nabla v_\varepsilon|^2 +\frac{\beta(t,x)}{2} \braket{v^2_\varepsilon,A(t,x)v^2_\varepsilon} \, dxdt  \\
& \le (1+D_3)\frac{C+ \frac{\tilde{M}_1}{\tilde{M}_2}}{\tilde{M}_2+D_3}e^{(\tilde{M}_2+D_3)T}+2\frac{\tilde{M}_1}{\tilde{M}_2}e^{\tilde{M}_2T}+\tilde{c}(T)=\mathcal{C}(T).
\end{align*}
In particular, we have
$$\int_0^T \int_\Omega |\nabla v_\varepsilon|^2  \, dxdt \le \frac{\mathcal{C}(T)}{\tilde{\rho}(T)+\mu}.$$
Finally, for every $T>0$,
$$\int_0^T \int_\Omega |v_\varepsilon|^2 \, dxdt \le k|\Omega|T.$$
Defining 
$$\bar{C}(T)=\max \left\{\frac{\tilde{c}(T)}{R_1},\mathcal{C}(T),k|\Omega|T \right\},$$ 
we obtain \eqref{e:step1-conclusion-1}, \eqref{e:step1-conclusion-2}, and \eqref{e:step1-conclusion-3}.
We observe that the constant $\bar{C}(T)$ does not depend on $D_4$ and $\mu$. \medskip

\noindent {\bf Step 2.} We now prove that \cref{uniestimnn1} holds without the additional assumption \eqref{e:step1-hypo-ag}.\medskip

Let us consider $d(t)=(d_1(t), \dots, d_k(t))$ which satisfies the condition \eqref{ip6} of \cref{hyponcoeff}.
Let us consider, for every $m \in \mathbb{N}$, $m \ge 1$:
$$d_m(t):=\frac{e^{-D_2t}}{e^{-D_2t}+\frac{1}{m}}d(t)+\frac{1}{m},$$
which satisfies
\begin{enumerate}[\rm(a)]
\item $d_m(t)\le mD_1+1$;\smallskip
\item $d_m(t)\ge (1-\frac{1}{me^{-D_2t}+1})\rho(t)+\frac{1}{m}$;
\item $d_m(t) \le (D_1+1)e^{D_2t}$;\smallskip
\item $|d'_m(t)| \le (D_2+D_3)d_m(t)$.
\end{enumerate}
First of all, we notice that, for every fixed $m\in\N$, $d_m$ satisfies the hypothesis \eqref{e:step1-hypo-ag} from Step 1 with constants $D_4=mD_1+1$ and $\mu=\frac1m$, and with 
$$\hat\rho(t):=\left(1-\frac{1}{e^{-D_2t}+1}\right)\rho(t).$$ Precisely, thanks to the estimates (a) and (b), we have
$$\hat\rho(t)+\mu\le \left(1-\frac{1}{me^{-D_2t}+1}\right)\rho(t)+\frac{1}{m}\le d_m(t)\le mD_1+1=D_4.$$
Second, we notice that, thanks to (c) and (d), $d_m$ satisfies the estimates from the hypothesis \eqref{ip6} of \cref{hyponcoeff} with constants $D_1+1$ and $D_2+D_3$ in place of $D_1$ and $D_3$.\smallskip 

\noindent Let $\varepsilon \in (0,\bar{\varepsilon})$ be fixed. For every $m \in \mathbb{N}$, let us consider $v_{\varepsilon,m} \in \mathcal{U}$, with $0 \le v_{\varepsilon,m} \le 1$, minimizer in $\mathcal U$ of the functional
\begin{equation*}
\mathcal{F}_{\varepsilon,m}(v):=\int_0^{+\infty} \int_\Omega \frac{e^{-\frac{t}{\varepsilon}}}{\varepsilon} \left\{ \varepsilon r(x) \cdot |\partial_t v|^2 + d_m(t) \cdot |\nabla v|^2 -2F(t,x,v) + \frac{\beta(t,x)}{2} \braket{v^2,A(t,x)v^2} \right\} \, dxdt.
\end{equation*}
By the previous step, for every $T>0$, the estimates \eqref{e:step1-conclusion-1}, \eqref{e:step1-conclusion-2}, \eqref{e:step1-conclusion-3} hold. Since the constant $\bar C(T)$ in \eqref{e:step1-conclusion-1}, \eqref{e:step1-conclusion-2}, \eqref{e:step1-conclusion-3} depends on the bounds on $d_m$ from \eqref{ip6} of \cref{hyponcoeff}, in which $D_1$ and $D_3$ are replaced by $D_1+1$ and $D_2+D_3$, we get that there is a constant $\check{C}(T)$, independent of $\varepsilon$ and $m$, such that 
\begin{equation}\label{e:step2-estimate-1}
\int_0^T \int_\Omega |\partial_t v_{\varepsilon,m}|^2 \, dxdt \le \check{C}(T);
\end{equation}
\begin{equation}\label{e:step2-estimate-2}
\int_0^T \int_\Omega |\nabla v_{\varepsilon,m}|^2 \, dxdt \le \frac{\check{C}(T)}{\displaystyle\inf_{t\in[0,T]}\hat\rho(t)+\frac{1}{m}};
\end{equation}
\begin{equation}\label{e:step2-estimate-3}
\int_0^T \int_\Omega |v_{\varepsilon,m}|^2 \, dxdt \le \check{C}(T).
\end{equation}
Now, by the definition of $\hat\rho$ and $\tilde\rho$ we have 
$$\inf_{t\in[0,T]}\hat\rho(t)\ge {\left(1-\frac{1}{e^{-D_2T}+1}\right)\inf_{t\in[0,T]}\rho(t)}= {\left(1-\frac{1}{e^{-D_2T}+1}\right)\tilde{\rho}(T)}.$$
Substituting in \eqref{e:step2-estimate-2}, we obtain
\begin{align*}
\int_0^T \int_\Omega |\nabla v_{\varepsilon,m}|^2 \, dxdt 
&\le \frac{\check{C}(T)}{(1-\frac{1}{e^{-D_2T}+1})\tilde{\rho}(T)+\frac{1}{m}}\\
&\le  \frac{\check{C}(T)}{(1-\frac{1}{e^{-D_2T}+1})\tilde{\rho}(T)},
\end{align*}
where $\tilde \rho$ is defined in \eqref{e:definition-rho-tilde}. \\
Moreover, by \cref{appeq}, $v_{\varepsilon,m}$ solves, for every $\varphi \in H^1_0(\Omega \times (0,T))$, for every $T>0$, and for every $i=1,\dots,k$:
\begin{align} \label{nap01}
\notag &\int_0^{+\infty} \int_\Omega \{ \varepsilon r_i(x)\partial_tv_{\varepsilon,m,i}\partial_t \varphi+r_i(x)\partial_t v_{\varepsilon,m,i}\varphi+d_{m,i}(t)\nabla v_{\varepsilon,m,i}\cdot \nabla \varphi -f_i(t,x,v_{\varepsilon,m,i})\varphi \\
& \qquad \qquad  +\beta(t,x)v_{\varepsilon,m,i}\sum_{j \ne i} a_{ij}(t,x)v^2_{\varepsilon,m,j}\varphi \} \, dxdt=0.
\end{align}
Thus, there exist a subsequence (not relabeled) $\set{v_{\varepsilon,m}}_m$ and a limit function $\bar{v}_\varepsilon \in \mathcal U$ such that, for every $T>0$
$$v_{\varepsilon,m} \xrightharpoonup[m \to +\infty]{} \bar{v}_\varepsilon \quad \text{in $H^1(\Omega \times (0,T))^k$ and a.e.}.$$
Moreover $0 \le \bar{v}_\varepsilon \le 1$, and, by weak convergence:
\begin{align*}
&\int_0^T \int_\Omega |\partial_t \bar{v}_\varepsilon|^2 \, dxdt \le \check{C}(T); \\ 
&\int_0^T \int_\Omega |\nabla \bar{v}_\varepsilon|^2 \, dxdt \le \frac{\check{C}(T)}{(1-\frac{1}{e^{-D_2T}+1})\tilde{\rho}(T)}; \\ 
&\int_0^T \int_\Omega |\bar{v}_\varepsilon|^2 \, dxdt \le \check{C}(T). 
\end{align*}
For every $\varphi \in H^1_0(\Omega \times (0,T))$, because $d_{m,i}(t) \le (D_1+1)e^{D_2T}$, and $d_{m,i}(t) \to d_i(t)$ a.e. for $m \to +\infty$, for all $i=1, \dots,k$, by dominate convergence we get 
$$d_{m,i}(t) \nabla \varphi \to d_i(t) \nabla \varphi \quad \text{strongly in $L^2(\Omega \times (0,T))^d$},$$
hence
$$\int_0^T \int_\Omega d_{m,i}(t)\nabla v_{\varepsilon,m,i}\cdot \nabla \varphi \, dxdt \to \int_0^T \int_\Omega d_i(t) \nabla \bar{v}_{\varepsilon,i} \cdot \nabla \varphi \,dxdt.$$
Similarly for the other terms, by the fact that for every $m$, $v_{\varepsilon,m}$ solves \eqref{nap01}, passing to the limit we get that $\bar{v}_\varepsilon$ solves in a weak sense \eqref{equaappr1} or \eqref{equaappr2}.
\end{proof}
At this point, we are ready to show the main result.
\begin{proof} [\bf Proof of \cref{existence}] 
Thanks to \cref{uniestimnn1}, there exist a subsequence $\set{\bar{v}_{\varepsilon_n}}$ and a function $\bar{w} \in \cap_{T>0} H^1(\Omega \times (0,T))^k$ such that
$$\bar{v}_{\varepsilon_n} \xrightharpoonup[n \to +\infty]{} \bar{w} \quad \text{in} \,\, H^1(\Omega \times (0,T))^k \,\, \forall T>0, \,\, \text{and a.e. in} \,\, \Omega \times \mathbb{R}^+.$$
Since $\mathcal{U}$ is convex and strongly closed, $\bar{w} \in \mathcal{U}$. Moreover $0 \le \bar{w} \le 1$. \\
By \cref{uniestimnn1}, for every $\varphi \in H^1_0(\Omega \times (0,T))$, for every $T>0$, for every $i=1,\dots,k$ and for every $n\in \mathbb{N}$, we have
\begin{align*}
&\int_0^{+\infty} \int_\Omega \{ \varepsilon_n r_i(x)\partial_t \bar{v}_{\varepsilon_n,i} \partial_t \varphi+ r_i(x)\partial_t \bar{v}_{\varepsilon_n,i}\varphi+d_i(t)\nabla \bar{v}_{\varepsilon_n,i}\cdot \nabla \varphi -f_i(t,x,\bar{v}_{\varepsilon_n,i})\varphi \\
&\qquad  \qquad +\beta(t,x)\bar{v}_{\varepsilon_n,i}\sum_{j \ne i} a_{ij}(t,x)\bar{v}^2_{\varepsilon_n,j}\varphi \} \, dxdt=0.
\end{align*}
Let us consider $\varphi \in  H^1_0(\Omega \times (0, \bar{T}))$ with null extension in $\Omega \times \mathbb{R}^+$. Let $i \in \set{1,\dots,k}$ be fixed. \\
Recalling the conditions of \cref{hyponcoeff}, $d_i(t) \in L^\infty((0,\bar{T}))$, and $r_i(x) \in L^\infty(\Omega)$: thus we get, for $\varepsilon_n \to 0$:
\begin{align*}
& r_i(x)\partial_t \bar{v}_{\varepsilon_n,i} \rightharpoonup r_i(x)\partial_t \bar{w}_i \quad \text{in} \,\, L^2(\Omega \times (0, \bar{T})); \\
& d_i(t) \nabla \bar{v}_{\varepsilon_n,i} \rightharpoonup d_i(t)\nabla \bar{w}_i \quad \text{in} \,\, L^2(\Omega \times (0, \bar{T})); \\
& f_i(t,x,\bar{v}_{\varepsilon_n,i}) \to f_i(t,x,\bar{w}_i) \quad \text{a.e.}; \\
& \bar{v}_{\varepsilon_n,i} \sum_{j \ne i} a_{ij}(t,x) \bar{v}^2_{\varepsilon_n,j} \to \bar{w}_i \sum_{j \ne i} a_{ij}(t,x)\bar{w}^2_j \quad \text{a.e.}.
\end{align*}
By weak convergence, and dominate convergence we obtain 
\begin{equation*}
\int_0^{+\infty} \int_\Omega r_i(x)\partial_t \bar{w}_{i}\varphi+d_i(t)\nabla \bar{w}_{i}\cdot \nabla \varphi -f_i(t,x,\bar{w}_{i})\varphi+\beta(t,x)\bar{w}_{i}\sum_{j \ne i} a_{ij}(t,x)\bar{w}^2_{j}\varphi \, dxdt=0.
\end{equation*}
\end{proof}
\begin{corollary}[Local H\"older continuity]
Let $w=(w_1,\dots w_k)$ be a solution found in \cref{existence}, then it is, as function in time, a $\frac{1}{2}$ locally H\"older function with respect to $||\cdot||_{L^2(\Omega)^k}$.
\end{corollary}
\begin{proof}
For all $T>0$ we get, by weak convergence, that there exists a constant $\hat{C}(T)$ such that
$$\int_0^T \int_\Omega |\partial_t w|^2 \, dxdt = \int_0^T ||\partial_t w||^2_{L^2(\Omega)^k} \, dt \le \hat{C}(T).$$
Let $t_1,t_2 \in [0,\bar{T}]$. \\
Because $w_i \in H^1((0,\bar{T}),L^2(\Omega))$ for every $i=1, \dots,k$,
$$w_i(t_2)=w_i(t_1)+\int_{t_1}^{t_2} \partial_t w_i(s) \, ds,$$ 
hence
$$||w_i(t_2)-w_i(t_1)||_{L^2(\Omega)} \le \int_{t_1}^{t_2} ||\partial_t w_i(s)||_{L^2(\Omega)} \, ds.$$
By Jensen inequality
$$||w(t_2)-w(t_1)||^2_{L^2(\Omega)^k}\le \sum_{i=1}^k \int_{t_1}^{t_2} ||\partial_t w_i(s)||^2_{L^2(\Omega)} \, ds |t_2-t_1| \le \hat{C}(\bar{T})|t_2-t_1|,$$
which implies:
\begin{equation*}
||w(t_2)-w(t_1)||_{L^2(\Omega)^k} \le \sqrt{\hat{C}(\bar{T})} |t_2-t_1|^{\frac{1}{2}}.
\end{equation*}
\end{proof}

\section{Asymptotic properties and segregation condition}

We want to show asymptotic properties of a weak solution of our parabolic system: assuming the conditions of \cref{hyponcoeff}, the purpose is to show that, if there exist a function $b:[0,+\infty) \to \mathbb{R}^+$ such that $\beta(t,x) \ge b(t) \,\, \forall \, (t,x) \in [0,+\infty) \times \Omega$, and $\lim_{t \to +\infty} b(t)=+\infty$, a constant $\bar{\mu}>0$ such that $a_{i,j}(t,x) \ge \bar{\mu} \,\, \forall \, i\ne j, \,\, \forall \, (t,x) \in [0,+\infty) \times \Omega$, and if $D_2,H_2=0$, then, if there exists the limit in time of our weak solution, this limit is a segregated function.

\begin{lemma} \label{sq}
Let us assume the conditions of \cref{hyponcoeff}. \\
Let $w=(w_1, \dots ,w_k) \in \cap_{T>0} H^1(\Omega \times (0,T))^k$ be a weak solution, in the sense of $H^{-1}(\Omega \times (0,T))^k$ for all $T>0$, of 
\begin{equation} \label{nnnn2}
\begin{cases}
r_i(x)\partial_t w_{i} -d_i(t)\Delta w_{i}=f_i(t,x,w_{i})-\beta(t,x)w_{i}\sum_{j\ne i}a_{ij}(t,x)w^2_{j} \quad \forall i=1,\dots,k, \\
0 \le w_{i} \le 1 \quad \forall i=1,\dots,k,
\end{cases}
\end{equation}
then, for every $T>0$, and for every $\delta>0$ and $\tau>0$, there exists a constant $C_{\delta}$ independent of $T$ and $\tau$ such that, for every $i=1,\dots,k$
\begin{align}
&\notag \int_T^{T+\tau} \int_{\Omega_\delta} d_i(t)|\nabla w_i|^2 +\beta(t,x) w^2_i \sum_{j \ne i} a_{ij}(t,x)w^2_j \, dxdt  \\
&\le C_{\delta}\biggl(1+\int_{T}^{T+\tau} d_i(t) \, dt+\int_{T}^{T+\tau} \int_\Omega |f_i(t,x,w_i)| \, dxdt\biggr),
\end{align}
where $\Omega_\delta=\set{x \in \Omega \,:\, d(x, \partial \Omega) > \delta}$. \\
Moreover, if $w$ is such that $w=0$ on $\partial \Omega$ for every $t$, there exists $C$ $\delta$-independent such that
\begin{equation*}
\int_T^{T+\tau} \int_{\Omega} d_i(t)|\nabla w_i|^2 +\beta(t,x) w^2_i \sum_{j \ne i} a_{ij}(t,x)w^2_j \, dxdt \le C\biggl(1+\int_{T}^{T+\tau} \int_\Omega |f_i(t,x,w_i)| \, dxdt\biggr).
\end{equation*}
\end{lemma}
\begin{proof}
Let us fix $\delta>0$, $\tau>0$ and $T>0$. 
Let us consider $0<\sigma<1$ such that $T>2\sigma$. \\
Let us take $\gamma_\delta=\gamma_\delta(x)$ a cut-off function of $\bar{\Omega}_\delta$ in $\Omega_{\frac{\delta}{2}}$. \\
Let us consider also a function $\alpha_{T,\tau}=\alpha_{T,\tau}(t)$ so defined:
\begin{equation*}
\alpha_{T,\tau}(t)= \begin{cases}
0 \quad &\text{in} \, [0,T-\sigma] \\
\frac{1}{\sigma}(t-(T-\sigma)) \quad &\text{in} \, (T-\sigma,T) \\
1 \quad &\text{in} \, [T,T+\tau] \\
-\frac{1}{\sigma}(t-(T+\tau+\sigma)) \quad &\text{in} \, (T+\tau,T+\tau+\sigma) \\
0 \quad &\text{in} \, [T+\tau+\sigma,+\infty). \\
\end{cases}
\end{equation*}
By the fact that $w$ is a weak solution of \eqref{nnnn2}, for every $\varphi \in H^1_0(\Omega \times (0,T))$, for every $T>0$, and for every $i=1,\dots,k$, we have
$$\int_0^{+\infty} \int_\Omega r_i(x)\partial_t w_{i}\varphi+d_i(t)\nabla w_{i}\cdot \nabla \varphi -f_i(t,x,w_{i})\varphi+\beta(t,x)w_{i}\sum_{j \ne i} a_{ij}(t,x)w^2_{j}\varphi \, dxdt=0.$$
Let us test, for every $i=1,\dots,k$, with $\varphi(t,x):=w_i(t,x)\alpha_{T,\tau}(t) \gamma^2_\delta(x)$. \\
Let us do the following computations:
\begin{align*}
\int_0^{+\infty} \int_\Omega r_i(x)\partial_t w_i \varphi \, dxdt &= \int_0^{+\infty} \alpha_{T,\tau}(t) \int_\Omega \partial_t(w_i\gamma_\delta \sqrt{r_i}) w_i\gamma_\delta \sqrt{r_i} \, dxdt \\
&= \frac{1}{2} \int_0^{+\infty} \alpha_{T,\tau}(t) \frac{d}{dt}||w_i\gamma_\delta \sqrt{r_i}||^2_{L^2(\Omega)} \, dt  \\
&=-\frac{1}{2} \int_0^{+\infty} \alpha'_{T,\tau}(t) ||w_i\gamma_\delta \sqrt{r_i}||^2_{L^2(\Omega)} \, dt  \\
&\ge -\frac{R_2|\Omega|}{2} \int_0^{+\infty} |\alpha'_{T,\tau}(t)| \, dt=-R_2|\Omega|;
\end{align*}
\begin{align*}
\int_0^{+\infty} \int_\Omega d_i(t)\nabla w_i \cdot \nabla \varphi \, dxdt &= \int_0^{+\infty} \alpha_{T,\tau}(t) d_i(t) \int_\Omega \gamma^2_\delta |\nabla w_i|^2 +2w_i \gamma_\delta \nabla w_i \cdot \nabla \gamma_\delta \, dxdt  \\
& \ge \int_0^{+\infty} \alpha_{T,\tau}(t)d_i(t) \int_\Omega \frac{1}{2}\gamma^2_\delta |\nabla w_i|^2 -2 |\nabla \gamma_\delta|^2 \, dxdt  \\
& \ge \frac{1}{2} \int_T^{T+\tau} \int_{\Omega_\delta} d_i(t)|\nabla w_i|^2 \, dxdt -2||\nabla \gamma_\delta||^2_{L^2(\Omega)}\int_{T-\sigma}^{T+\tau+\sigma} d_i(t) \, dt;
\end{align*}
\begin{equation*}
\int_0^{+\infty} \int_\Omega f_i(t,x,w_i)\varphi \, dxdt \le \int_{T-\sigma}^{T+\tau+\sigma} \int_\Omega |f_i(t,x,w_i)| \, dxdt;
\end{equation*}
\begin{equation*}
\int_0^{+\infty} \int_\Omega \beta(t,x)w_{i}\sum_{j \ne i} a_{ij}(t,x)w^2_{j}\varphi \, dxdt  \ge \int_T^{T+\tau} \int_{\Omega_\delta} \beta(t,x)w^2_{i}\sum_{j \ne i} a_{ij}(t,x)w^2_{j} \, dxdt.
\end{equation*}
Hence, for $\sigma \to 0$ we get:
\begin{align*}
&\int_T^{T+\tau} \int_{\Omega_\delta} d_i(t)|\nabla w_i|^2 +\beta(t,x) w^2_i \sum_{j \ne i} a_{ij}(t,x)w^2_j \, dxdt  \\
&\le R_2|\Omega|+2||\nabla \gamma_\delta||^2_{L^2(\Omega)}\int_{T}^{T+\tau} d_i(t) \, dt+\int_{T}^{T+\tau} \int_\Omega |f_i(t,x,w_i)| \, dxdt \\
&\le C_{\delta}(1+\int_{T}^{T+\tau} d_i(t) \, dt+\int_{T}^{T+\tau} \int_\Omega |f_i(t,x,w_i)| \, dxdt),
\end{align*}
where $C_{\delta}:=2R_2|\Omega|+4||\nabla \gamma_\delta||^2_{L^2(\Omega)}+2$. \\
If $w=0$ on $\partial \Omega$, taking $\varphi(t,x):=w_i(t,x)\alpha_{T,\tau}(t)$, and similarly to the previous passages we obtain the thesis.
\end{proof}
\begin{corollary} \label{sq2}
Under the same hypotheses of \cref{sq}, let us assume also that $D_2,H_2=0$ in \eqref{ip2} and \eqref{ip6} of \cref{hyponcoeff}. Then for every $T>0$, and for every $\delta>0$ and $\tau>0$, there exists a constant $C_{\delta,\tau}$ independent of $T$ such that, for every $i=1,\dots,k$
\begin{equation} \label{nnnn3}
\int_T^{T+\tau} \int_{\Omega_\delta} d_i(t)|\nabla w_i|^2 +\beta(t,x) w^2_i \sum_{j \ne i} a_{ij}(t,x)w^2_j \, dxdt \le C_{\delta,\tau},
\end{equation}
where $\Omega_\delta=\set{x \in \Omega \,:\, d(x, \partial \Omega) > \delta}$. \\
Moreover, if $w$ is such that $w=0$ on $\partial \Omega$ for every $t$, the constant is $\delta$-independent, and \eqref{nnnn3} is true with $\Omega$ instead of $\Omega_\delta$.
\end{corollary}
\begin{proof}
By the fact that $D_2,H_2=0$, we get
\begin{align*}
&\int_{T}^{T+\tau} d_i(t) \, dt \le D_1\tau; \\
&\int_{T}^{T+\tau} \int_\Omega |f_i(t,x,w_i)| \, dxdt \le H_1||U_2||\tau.
\end{align*}
So we obtain the thesis.
\end{proof}
\begin{proof} [\bf Proof of \cref{lsp}]
Let us fix $i \in \set{1,\dots,k}$. \\
By \cref{sq2}, for every $\delta>0$ there exists a constant $C_{\delta,1}>0$ such that
$$\int_T^{T+1} \int_{\Omega_\delta} \beta(t,x) w^2_i \sum_{j \ne i} a_{ij}(t,x)w^2_j \, dxdt \le C_{\delta,1}.$$
For $T$ sufficiently large, such that $b(t) \ge \eta >0$ for all $t \ge T$,
$$\int_T^{T+1} \int_{\Omega_\delta} w^2_i \sum_{j \ne i} \bar{\mu}w^2_j \, dxdt \le \frac{C_{\delta,1}}{\inf_{[T,T+1]}b(t)}.$$
Let us fix $\delta$. \\
Let us define
$$V_T:= \biggl\{t \in[T,T+1] \,: \, \int_{\Omega_\delta} w^2_i \sum_{j \ne i} \bar{\mu}w^2_j \, dx \le \frac{4C_{\delta,1}}{\inf_{[T,T+1]}b(t)}\biggr\} .$$
Let us observe that:
$$|V_T| \ge \frac{1}{2}.$$
If this were not true, we would have that 
$$\int_T^{T+1} \int_{\Omega_\delta} w^2_i \sum_{j \ne i} \bar{\mu}w^2_j \, dxdt \ge \frac{2C_{\delta,1}}{\inf_{[T,T+1]}b(t)},$$ 
that is false. \\
Let us take, for $n$ sufficiently large, a sequence $\set{t_n}_n$ such that $t_n \in V_n$. \\
The fact that $w_i(t_n) \to_{n \to +\infty} \bar{v}_i$ in $L^1(\Omega_\delta)$ implies that there exists a subsequence $\set{t_{n_h}}_h$ such that
$$w_i(t_{n_h}) \xrightarrow[h \to +\infty]{} \bar{v}_i \quad \text{a.e. in} \,\, \Omega_\delta.$$
Hence, by Fatou lemma
\begin{align*}
\int_{\Omega_\delta} \bar{v}^2_i \sum_{j \ne i} \bar{\mu} \bar{v}^2_j \, dx & \le \liminf_{h \to +\infty}\int_{\Omega_\delta}  w^2_i(t_{n_h}) \sum_{j \ne i} \bar{\mu} w^2_j(t_{n_h}) \, dx \\
& \le  \liminf_{h \to +\infty} \frac{4C_{\delta,1}}{\inf_{[t_{n_h},t_{n_h}+1]} b(t)}=0.
\end{align*}
Thus $\bar{v}^2_i \sum_{j \ne i} \bar{\mu} \bar{v}^2_j=0$ a.e. in $\Omega_\delta$, which implies
$$\bar{v}_i\bar{v}_j=0 \quad \text{in} \,\, \Omega_\delta, \quad \forall \, j \ne i.$$
For the arbitrariness of $\delta$ we get
$$\bar{v}_i\bar{v}_j=0 \quad \text{in} \,\, \Omega, \quad \forall \, j \ne i.$$
Because it is true for all $i=1,\dots,k$, we obtain the thesis.
\end{proof}

\section*{Acknowledgements} This work was partially supported by the European Research Council (ERC), under the European Union's Horizon 2020 research and innovation program, through the project ERC VAREG - {\em Variational approach to the regularity of the free boundaries} (No.\,853404). The author would like to thank Bozhidar Velichkov for the discussions.

\bibliographystyle{aomalpha} 

\end{document}